\documentclass[a4paper,12pt]{article}
\usepackage[T1]{fontenc}
\usepackage{amsmath, mathtools, graphicx, color, amsthm, amssymb,tikz}
\usepackage[font=small,labelfont=bf,width=12.5cm]{caption}
\usepackage{algorithmic, algorithm}
\usepackage{MnSymbol}

\usepackage{hyperref}

\newtheorem{theorem}{Theorem}[section]

\newtheorem{corollary}[theorem]{Corollary}
\newtheorem{proposition}[theorem]{Proposition}
\newtheorem{conjecture}[theorem]{Conjecture}
\newtheorem{question}[theorem]{Question}
\newtheorem{problem}[theorem]{Problem}

\newtheorem{observation}{Observation}

\title{Remarks on odd colorings of graphs}

\author
{
    Yair Caro \thanks{University of Haifa-Oranim, Israel. E-mail: \texttt{yacaro@kvgeva.org.il}},
    \quad
	Mirko Petru\v{s}evski \thanks{Department of Mathematics and Informatics, Faculty of Mechanical Engineering - Skopje, Republic of North Macedonia. E-Mail: \texttt{mirko.petrushevski@gmail.com}},
	\quad
	Riste \v{S}krekovski\thanks{FMF, University of Ljubljana \& Faculty of Information Studies, Novo mesto, Slovenia. E-Mail: \texttt{skrekovski@gmail.com}}
}

\begin{document}
\maketitle

\begin{abstract}
 A proper vertex coloring $\varphi$ of graph $G$ is said to be odd if for each non-isolated vertex $x\in V(G)$ there exists a color $c$ such that $\varphi^{-1}(c)\cap N(x)$ is odd-sized. The minimum number of colors in any odd coloring of $G$, denoted $\chi_o(G)$, is the odd chromatic number. Odd colorings  were recently introduced in [M.~Petru\v{s}evski, R.~\v{S}krekovski: \textit{Colorings with neighborhood parity condition}]. Here we discuss various basic properties of this new graph parameter, characterize acyclic graphs and hypercubes in terms of odd chromatic number, establish several upper bounds in regard to degenericity or maximum degree, and pose several questions and problems.
  \end{abstract}

\medskip


%-----------------------------------------------------------------------------------------------------------------------
%-----------------------------------------------------------------------------------------------------------------------

\section{Introduction}

 All considered graphs are simple, finite and undirected. We follow~\cite{BonMur08} for any terminology and notation not defined here. A $k$-(vertex-)coloring of a graph $G$ is an assignment $\varphi: V(G)\to\{1,\ldots,k\}$. A coloring $\varphi$ is said to be \textit{proper} if every color class is an independent subset of the vertex set of $G$. A \textit{hypergraph} $\mathcal{H}=(V(\mathcal{H}),\mathcal{E}(\mathcal{H}))$ is a generalization of a graph, its (hyper-)edges are subsets of $V(\mathcal{H})$ of arbitrary positive size. There are various notions of (vertex-)coloring of hypergraphs, which when restricted to graphs coincide with proper graph coloring. One such notion was introduced by Cheilaris et al.~\cite{CheKesPal13}. An \textit{odd coloring} of hypergraph $\mathcal{H}$ is a coloring such that for every edge $e\in \mathcal{E}(\mathcal{H})$ there is a color $c$ with an odd number of vertices of $e$ colored by $c$. Particular features of the same notion notion (under the name \textit{weak-parity coloring}) have been considered by Fabrici and G\"{o}ring~\cite{FabGor16} (in regard to face-hypergraphs of planar graphs) and also by Bunde et al.~\cite{BunMilWesWu07} (in regard to coloring of graphs with respect to paths, i.e., path-hypergraphs).

 Of recent interest are also edge-colorings of graphs with certain parity condition required at the vertices. We refer the reader to~\cite{AtaPetSkr16, KanKatVar18, LuzPetSkr15, Pet18} for edge-colorings which require an odd number of occurrences of every color that appears at a vertex; edge-colorings with the weaker assumption that at least one color per vertex has an odd number of occurrences are studied in~\cite{Pet15}. The natural generalizations obtained by assigning a parity signature to each vertex and asking that the obtained parity condition is fulfilled by every (resp. some) color appearing at a vertex are considered in~\cite{LuzPetSkr18}. The survey~\cite{PetSkr21} deals with analogous covering aspects.

 In this paper we study certain aspects of odd colorings for graphs with respect to (open) neighborhoods, that is, the colorings of graph $G$ such that for every non-isolated vertex $x$ there is a color that occurs an odd number of times in the neighborhood $N_G(x)$. Our focus is on colorings that are at the same time proper.

As defined in~\cite{PetSkr22}, a proper coloring of a graph $G$ is {\em odd} if in the open neighborhood $N(v)$ of every non-isolated vertex $v$ a color
appears an odd number of times. Denote by $\chi_o(G)$ the minimum number of colors in any odd coloring of $G$, and call this the \textit{odd chromatic number} of $G$. Note that the obvious inequality $\chi(G)\leq\chi_o(G)$ becomes an equality whenever the graph $G$ is \textit{odd}, that is, if it has only odd vertex degrees. On the other hand, the mentioned inequality may also be strict. In fact, the ratio $\chi_o(G)/\chi(G)$ can acquire arbitrarily high values. Indeed, consider a non-empty graph $H$ and let $G$ be the graph obtained from $H$ by subdividing every edge in $E(H)$ once; that is, let $G$ be the \textit{complete subdivision} of $H$. Then $\chi_o(G)\geq\chi(H)$ whereas $\chi(G)=2$. In particular, for every $n\geq2$, the complete subdivision of $K_n$ has odd chromatic number $\geq n$ and (ordinary) chromatic number $2$.

Given a graph $G$, let $L$ be a function which assigns to each vertex $v$ of $G$ a set $L(v)$, called the \textit{list} of $v$. An odd coloring $c$ of $G$ such that $c(v)\in L(v)$ for all $v\in V(G)$ is called an \textit{odd list coloring} of $G$ with respect to $L$, or an \textit{odd $L$-coloring}, and we say that $G$ is \textit{odd $L$-colorable}. A graph $G$ is said to be \textit{odd $k$-list-colorable} if it has an odd list coloring whenever all the lists are of length $k$. Every graph is clearly odd $n$-list colorable, where $n$ is the order of $G$. The smallest value of $k$ for which
$G$ is odd $k$-list colorable is called the \textit{odd list chromatic number} (or \textit{odd choice number}) of $G$, denoted ${\rm ch}_o(G)$. Note that the obvious inequality $\chi_o(G)\leq{\rm ch}_o(G)$ can be strict. In fact, the ratio ${\rm ch}_o(G)/\chi_o(G)$ can be arbitrarily large. Indeed, for any odd positive integer $n$ it holds that ${\rm ch}_o(K_{n,n^n})=n+1$ whereas $\chi_o(K_{n,n^n})=2$.

\smallskip

The paper is organized as follows. The next section collects several basic observations for the odd chromatic number. In Section~3 we characterize acyclic graphs and hypercubes in terms of the same graph parameter. This is followed by a section on general upper bounds. At the end comes a short Section~5, where we briefly convey several ideas for possible further work on the topic of odd colorings of graphs.

%--------------------------------------------------------------------------
\section{Basic properties of the odd chromatic number}

Here we compare the behavior of the odd chromatic number to the behavior of the (ordinary) chromatic number in regard to standard graph notions and concepts. The followings observations are rather easy, but nevertheless
their further improvements will give new interesting results and certainly will help for better understanding of this new concept.

%-------------------------------------------------------------------------
\paragraph{Complexity.} We start this section with a brief discussion concerning the complexity of determining the odd chromatic number.

\begin{observation}
The problem of determining $\chi_o(G)$ for a graph $G$ is  NP-hard.
\end{observation}

\begin{proof}
Since determining the chromatic number of a graph is an NP-hard problem, the same
holds for the odd  chromatic number. Reduction is obvious, to a given graph $H$ attach a leaf to every vertex of positive even degree, and denote the resulting odd graph by $G$. Obviously, $\chi(H)=\chi(G)=\chi_o(G)$.
\end{proof}

Note that graphs with chromatic number $\le 2$ are precisely bipartite graphs, but regarding odd colorings a graph without isolated vertices has odd chromatic number $2$ if and only if it is bipartite with all vertices of odd degree.
Also notice that there is no non-empty graph with odd chromatic number exactly 1.

\paragraph{Bridges.}  Supposing a given graph has a non-trivial bridge, the usual chromatic number is the maximum of the chromatic numbers of both sides. Regarding the odd chromatic number, things are slightly different.

\begin{observation}
Suppose $G$  has a non-trivial bridge $e =uv$ with  $G_1$ and $G_2$ being the two parts of $G-e$. Then
        \begin{equation}
           \chi_o(G)  \le  \max \{3, \chi_o(G_1) , \chi_o(G_2)\}.
        \end{equation}
 \end{observation}
 \begin{proof}
 We may assume $ \chi_o(G_2) \ge  \chi_o(G_1)$ and that  $u$  is in $G_1$ and $v$ is in $G_2$.
 Take an odd coloring of  $G_1$ with the colors $1,\ldots,\chi_o(G_1)$  such  that the color of $u$ is $1$.
 Also, take an odd coloring of $G_2$ with the colors $1,2,\ldots,\chi_o(G_2)$ such that $v$ gets color 2.

If we have used at least three colors for the coloring of $G_2$, then we can easily achieve that the color  of $v$ is also distinct  from a color that appears an odd number of times in $N_{G_1}(u)$ and that the color of $u$ is distinct from a color that appears an odd number of times in $N_{G_2}(v)$.  And then so the constructed coloring of $G$ is odd $G$ and uses $\max \{\chi_o(G_1) , \chi_o(G_2)\}\ge 3$ colors.

Otherwise, $\chi_o(G_1) =\chi_o(G_2)=2$. Observe that then it is not possible to odd color $G$ with just two colors. But in this case
we simply alter the initial coloring of $G_2$ by  recoloring all vertices of color 2 by 3 and all vertices of color 1 by 2. The modified coloring
becomes odd for $G$.
 \end{proof}

The above bound is obviously tight for many pairs of graphs $G_1,G_2$ where at least one of them
has odd chromatic number $\ge 3$. In the case $\chi_o(G_1) =\chi_o(G_2)=2$, the above bound is tight always
whenever both $G_1$ and $G_2$ are bipartite odd graphs. Also observe that it may happen that
 $\chi_o(G)  \le  \min \{\chi_o(G_1) , \chi_o(G_2)\}$, it happens for example when $G_1$ and $G_2$ are cycles
 of particular order.

%-------------------------------------------------------------------------
\paragraph{Introducing/removing a vertex.} By introducing a new vertex to an existing graph, the (ordinary) chromatic number either stays the same or increases by $1$. There is a similar upper bound for the odd chromatic number of a graph in terms of the odd chromatic number of a vertex-deleted subgraph.

 \begin{observation}
 \label{o.newvertex}
Let $G$ be a graph without isolated vertices.  Introduce a new vertex $v$ connected to some/all vertices of $G$ to obtain a
graph $H$. Then,  $\chi_o(H) \le \chi_o(G)+2$.
\end{observation}
 \begin{proof}
Consider an odd coloring of $G$ with $\chi_o(G)$ colors. Choose a neighbor $w$ of $v$ and give $w$ and $v$ a pair of fresh new colors.  So $v$ has the color of $w$  uniquely in $N(v)$, and $w$ has the color of $v$ uniquely in $N(w)$,  and all neighbors
of $w$ have the color of $w$ appearing uniquely in their neighborhoods.
 \end{proof}

The above bound can be attained for  $P_3$ with the new vertex  $v$ being adjacent to the ends of the path to obtain $C_5$. In view of Observation~\ref{o.newvertex}, by introducing a new vertex the odd chromatic number cannot increase by much
(at most by $2$).

\smallskip

As already observed in~\cite{PetSkr22}, the parameter $\chi_o$ is not monotonic in regard to the subgraph relation. One naturally wonders whether by removing a vertex the odd chromatic number can increase significantly.

\begin{observation}
With the notation of Observation~\ref{o.newvertex}, the difference $\chi_o(G)-\chi_o(H)$ can be arbitrarily large.
\end{observation}
\begin{proof}
Let $G$ be a bipartite graph of odd order and without isolated vertices. Note that $G$ has an odd number of vertices with even degree. Add a new vertex and connect it by an edge to every vertex in $G$ of even degree. The obtained graph $H$ is $3$-colorable and has only odd vertex
degrees. Hence $\chi_o(H) \leq 3$. In spite of this, $\chi_o(G)$ can be arbitrarily large. For example, take $G$ to be the graph obtained by subdividing once every edge of $K_n$, where $n\equiv1$ or $2\, (\text{mod }4)$; thus $G$ is bipartite of order $\binom{n+1}{2}\equiv1 \, (\text{mod }2)$ and $\chi_o(G)= n$.
\end{proof}

%-------------------------------------------------------------------------
\paragraph{Disjoint union.} Here we briefly state a partition property concerning the behavior of the odd chromatic number when $V(G)$ admits a representation as a disjoint union $A\cup B$, where the induced subgraphs on $A$ and $B$ are \textit{isolate-free}, that is, are free from isolated vertices. (On the existence of such partitions already in $2$-connected graphs see e.g.~\cite{HoyTho16}.)

\begin{observation}
\label{o.2parts}
Let $V(G)  =  A  \cup  B$ is a disjoint union, where the induced subgraphs $G[A]$ and $G[B]$ are isolate-free.  Then
\begin{equation}
\chi_o(G) \le \chi_o(G[A]) + \chi_o(G[B])\,.
\end{equation}
\end{observation}

\begin{proof}
Color the vertices in $A$ with a set of $\chi_o(G[A])$ colors so as to obtain an (optimal) odd coloring of $G[A]$. Using a disjoint set of $\chi_o(G[B])$ colors, color the vertices in $B$ so as to produce an odd coloring of $G[B]$. Since $A$ and $B$ are isolate-free, the constructed coloring of $G$ is odd and we have used $\chi_o(G[A]) + \chi_o(G[B])$ colors in total.
\end{proof}
Notice that the bound is sharp, e.g. it is achieved by $C_4$.
From Observation~\ref{o.2parts} we immediately have the following for the join operation $G \vee H$ when $G$ and $H$ are isolate-free:

\begin{observation}
If $G$ and $H$ are nontrivial connected graphs then
\begin{equation}
\chi(G) + \chi(H) \le \chi_o(G\vee H) \le \chi_o(G) + \chi_o(H)\,.
\end{equation}
\end{observation}

\begin{proof}
Consider $V(G)$ and $V(H)$ as the parts in the partition of $V(G\vee H)$.
\end{proof}

%-------------------------------------------------------------------------

\paragraph{Cartesian product.} Next we consider the behavior of the odd chromatic number under taking cartesian products. Regarding the usual coloring of cartesian product graphs, there is a beautiful relation, see~\cite{Viz63,Abe64}, which claims
\begin{equation}
\label{eq.cartesian_bound}
\chi(G\Box H)= \max \{ \chi(G) ,\chi(H)\}.
\end{equation}
The analogous equality for the odd chromatic number is not true in general. For example, $K_2\Box Q_3=Q_4$ and $\chi_o(K_2)=\chi_o(Q_3)=2$ whereas $\chi_o(Q_4)=4$ (cf. Theorem~\ref{hypercube} in the next section). For the odd chromatic number of a cartesian product we have the following
much weaker bound.
 \begin{observation}
 If $G,H$ are connected nontrivial graphs, then
      \begin{equation}
             \label{in.cartesian_bound}
              \chi_o(G \Box H)  \le \min\{\chi(G)\cdot\chi_o(H), \chi_o(G)\cdot\chi(H)\} \le \chi_o(G) \cdot \chi_o(H)\,.
      \end{equation}
 \end{observation}
\begin{proof}
By symmetry, it suffices to prove that $\chi_o(G \Box H)\le \chi(G)\cdot\chi_o(H)$.
Let $g$ be a proper coloring of $G$  with $\chi(G)$ colors, and let $h$ be an odd coloring of $H$  with $\chi_o(H)$ colors that are all distinct from those of $G$.
To any vertex  $v = (u,w)\in G \Box H$,  we assign the pair $f(v) =  f(u,w)  = (g(u),h(w))$. We prove that $f$ is an odd coloring of $G \Box H$ with $\chi(G)\cdot\chi_o(H)$ colors.

Consider the neighborhood $N(v)$. It is comprised of vertices of two kinds: either of the form $(u,w')$ where $w'\in N_H(w)$,  or of the form $(u',w)$ where $u'\in N_G(u)$.
Since the colorings $g$ and $h$ are proper, the set of colors used for the vertices of the first kind is disjoint from the set of colors used for the vertices of the second kind. So it suffices to observe that a color has an odd number of occurrences on the vertices of the first kind, as $h$ is an odd coloring of $H$.
Thus $f$ is indeed an odd coloring of $G\Box H$ with $\chi(G)\cdot\chi_o(H)$ colors.
\end{proof}

The above bound is attained for $G=H=K_2$. For an infinite family of graphs $G$ such that $\chi_o(G \Box K_2)= \chi_o(G) \cdot \chi_o(K_2)$ see Theorem~\ref{hypercube} and the after remark in the next section. We are optimistic regarding the following problem.

\begin{problem}
Improve the bounds of (\ref{in.cartesian_bound}) to a bound comparable to the right side of (\ref{eq.cartesian_bound}).
\end{problem}

%-------------------------------------------------------------------------
\paragraph{Domination.} We have already mentioned in the introduction that the ratio $\chi_o(G)/\chi(G)$ can be arbitrarily large. Consequently, so can the difference $\chi_o(G)-\chi(G)$. Nevertheless, by using the monotonicity of the (ordinary) chromatic number, we can easily show that $\chi_o(G)-\chi(G)$ never exceeds the total domination number $\gamma_t(G)$.

\begin{observation}
Every graph $G$ satisfies
       \begin{equation}
           \label{b.t}
            \chi_o(G) \le \gamma_t(G) + \chi(G)\,.
       \end{equation}
\end{observation}
\begin{proof}
Color each vertex from an optimal total dominating set $S$ with a new color. To the remaining graph $G-S$ apply a proper coloring by new
$\chi(G-S)\le \chi(G)$ colors. This way we use at most $\gamma_t(G)+\chi(G)$ colors to properly color $G$
and each vertex of $G$ sees some color uniquely, namely a color of a vertex of $S$ by which it is dominated.
\end{proof}

In fact, the total domination number $\gamma_t(G)$  in (\ref{b.t}) can be replaced by the so-called \textit{even domination number}, $\gamma_{e}(G)$, defined as follows. We say that $S\subseteq V(G)$  is an \textit{even dominating set},  if  for every vertex  $v$ of positive even degree in $G$, $N_G(v)\cap S\neq \emptyset$. Note that if $v\in S$ is of positive even degree in $G$ then there is another vertex in $S$ (not necessarily of even degree) that dominates it. Let $\gamma_{e}(G)$ be the minimum size of an even dominating set. Observe that $\gamma_{e}(G)\leq\gamma_t(G)$ and also $\gamma_{e}(G)\leq n_e(G)$, where $n_e(G)$ is the number of non-isolated vertices of even degree in $G$. To our knowledge this domination number was not defined
before and it can be of independent interest.

\begin{observation}
    \label{domination}
Every graph $G$ satisfies
\begin{equation}
     \chi_o(G) \leq\gamma_{e}(G) + \chi(G) \leq \min\{n_e(G)+\chi(G), \gamma_{t}(G) +\chi(G)\}\,.
\end{equation}
\end{observation}

\begin{proof}
Let $S$ be a total even dominating set of size $\gamma_{e}(G)$  and color all the vertices of $S$ with distinct colors. Consider  $H = G - S$. Since $\chi(H)\leq\chi(G)$, properly color $H$ with at most $\chi(G)$ colors. The constructed coloring of $G$  is odd.
Indeed, properness is clear. As for odd occurrence of a color in each non-empty vertex neighborhood, we need to care only for (non-isolated) vertices $v$ of even degree.  Such a vertex $v$ is dominated by a vertex from $S$, which has a unique color.  Altogether we use at most $\gamma_{e}(G) + \chi(G)$ colors. Now the second inequality follows by the remarks stated
in front of this observation.
\end{proof}

The obtained bound is achieved for even stars $K_{1,2m}$  as $\gamma_{e}(K_{1,2m})= 1$ (while $\gamma_t(K_{1,2m}) =2$)  and the bound $\chi_o(K_{1,2m})=  \gamma_{e}(K_{1,2m}) +\chi(K_{1,2m}) = 3$ is sharp.
Also for odd graphs $G$ we have $\gamma_{e}(G) = 0$  and $\chi_o(G) = \chi(G)$.

%\newpage

%---------------------
\section{Characterizations in terms of $\chi_o$}

We begin this section by observing that paths and cycles are easily characterized in terms of their odd chromatic number. It is readily checked that for paths holds:

\begin{equation}
    \label{paths}
\chi_o(P_n)=
\begin{cases}
n & \text{\quad if \,} n\leq2 \,;\\
3 & \text{\quad if \,} n\geq3 \,.
\end{cases}
\end{equation}

\noindent And similarly, for cycles we have:

\begin{equation}
    \label{cycles}
\chi_o(C_n)=
\begin{cases}
3 & \text{\quad if \,} 3 \mid n \,;\\
4 & \text{\quad if \,} 3 \nmid n \text{ and } n\neq5\,;\\
5 & \text{\quad if \,} n=5 \,.
\end{cases}
\end{equation}

\noindent Let us now characterize trees.

\begin{proposition}
    \label{trees}
If $T$ is a non-trivial tree then
\begin{equation*}
\chi_o(T)=
\begin{cases}
2 & \text{\quad if \,} T \text{ is odd} \,;\\
3 & \text{\quad otherwise}\,.
\end{cases}
\end{equation*}
\end{proposition}

\begin{proof}
 Note that $T$ is odd if and only if $\chi_o(T)=\chi(T)=2$, since $T$ is non-trivial and bipartite. It remains to prove that every tree admits an odd $3$-coloring. This is easily seen to hold true if $T$ is a star. For general trees we induct on the order $n(T)$. Let $T$ be a tree that is not a star. Consider a maximum path $P$ in $T$. Let $v$ and $w$ be, respectively, the second and third vertex of a traversal of $P$. Then the set $S=N_T(v)\backslash\{w\}$ consists of leaves (i.e. $v$ is an internal leaf of $T$). By induction, there exists an odd $3$-coloring of $T-S$. We may assume that $v$ and $w$ are colored by $1$ and $2$, respectively. Extend to $T$ by using the color $3$ for $S$.
\end{proof}

\smallskip

\noindent\textbf{Remark.} The proof of Proposition~\ref{trees} can be easily modified into proving that the odd list chromatic number of any tree equals its odd chromatic number. Namely, it amounts to observing the following: if all but one leaf of a star is already colored in a way that the obtained partial coloring is odd, then this coloring extends to the remaining vertex provided that we have a choice of three colors for it.

\medskip

Next we give an analogous characterization for hypercubes. It turns out that the value $\chi_o(Q_n)$ depends solely on the parity of $n$. Our findings are summarized in the following.

\begin{theorem}
    \label{hypercube}
If $Q_n$ is the $n$-dimensional cube then

\begin{equation*}
\chi_o(Q_n)=
\begin{cases}
2 & \text{\quad if \,} n \text{ is odd} \,;\\
4 & \text{\quad if \,} n \text{ is even}\,.
\end{cases}
\end{equation*}
\end{theorem}

\begin{proof}
The case of odd $n$ is rather trivial. Indeed, as then $Q_n$ is odd and bipartite, it holds that $\chi_o(Q_n)=\chi(Q_n)=2$. Assuming $n$ is even, the equality $Q_n=Q_{n-1} \Box K_2$ and Observation~$7$ give $\chi_o(Q_n)\leq\chi_o(Q_{n-1})\cdot\chi_o(K_2)=4$. We proceed to show that $\chi_o(Q_n)=4$. Let us look at $Q_n$ as if comprised of two copies of $Q_{n-1}$ with a perfect matching between them, and call any pair of end-vertices of a matching edge `allies'; we shall denote the ally of a vertex $x$ by $\bar{x}$.

Arguing by contradiction, suppose that an odd $3$-coloring $c$ of $Q_n$ exists (still assuming $n$ is even). Note that in each vertex neighborhood $N_{Q_n}(v)$ each of the two colors $\neq c(v)$ occurs an odd number of times. Consider the inherited proper $3$-coloring of the first copy of $Q_{n-1}$. Let $v$ be an arbitrary vertex from this copy, and let it be assigned with the color $1$ whereas the ally $\bar{v}$ of $v$ is colored by $2$. Let $U$ be the intersection of $N_{Q_{n-1}}(v)$ with the color class $[2]=c^{-1}(2)$. Thus $U$ is an even-sized set (possibly empty). Similarly, denote by $W$ the intersection of $N_{Q_{n-1}}(v)$ with the color class $[3]$, hence $W$ is an odd-sized set. Finally, let $V$ be comprised of all second-neighbors of $v$ within the first copy of $Q_{n-1}$ that are colored by $1$. Before proceeding further, let us recall the following important property of any hypercube: every pair of vertices in $N(v)$ have a unique neighbor in the second neighborhood of $v$, and conversely  every vertex  $z$ in the second neighborhood of $v$  is adjacent to exactly two vertices in $N(v)$.

\begin{figure}[ht!]
	$$
		\includegraphics[scale=0.7]{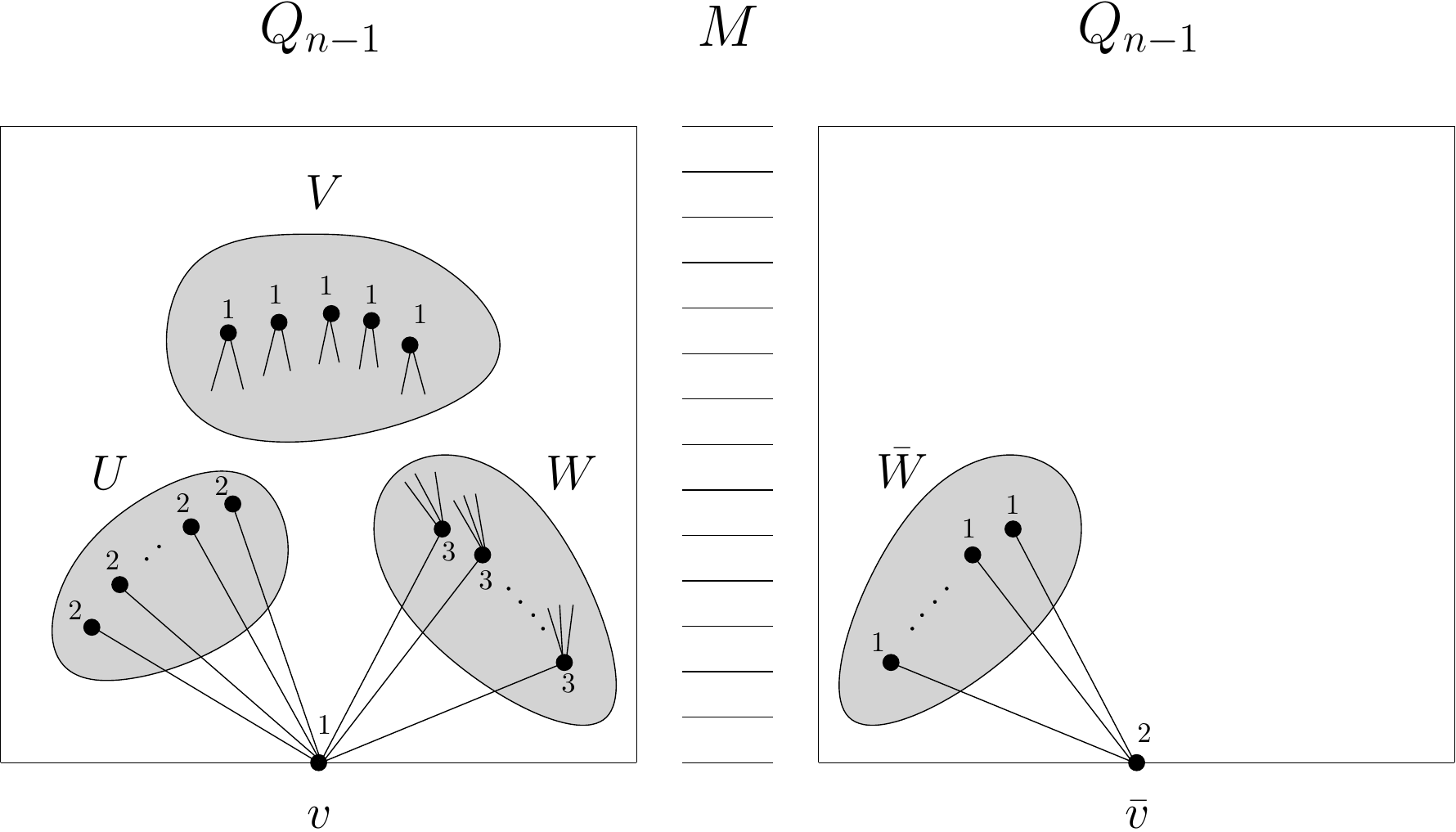}
	$$
	\caption{The sets $U,V$ and $W$. Every vertex from $V$ has two edges going down to $U\cup W$. Every edge from $W$ has an odd number of edges going up to $V$.}
	\label{fig:rules}
\end{figure}

Notice that for every $w\in W$, its ally $\bar{w}$ is adjacent with $\bar{v}$, implying that $\bar{w}$ is colored by $1$. This forces that $V$ is not empty as for each vertex $w$ there are already two neighbors of color $1$ (which are $v$ and $\bar{w}$). Consequently, in the bipartite graph $G=[V,W]$ induced by $V\cup W$, the degree of every $w$ is odd. Thus the size $e(G)=e([V\cup W])$ can be expressed as a sum of an odd number of odd summands, hence it is an odd number. Observe that for every $z\in V$, there are exactly two edges going from $z$ to $U\cup W$. This gives a partition of $V$ into three (possibly empty) subsets $V_{22}, V_{23}, V_{33}$, where $V_{22}$ (resp. $V_{33}$) collects those $z$'s that neighbor with two members of $U$ (resp. $W$) and $V_{23}$ is comprised of all $z$'s having one neighbor in $U$ and another one in $W$. Clearly, every vertex of $V_{22}$ is isolated in $G$, every vertex of $V_{23}$ is a leaf in $G$, and every vertex of $V_{33}$ has degree $2$ in $G$. Hence counting edges in $G$ from the $V$ side we get $|V_{23}|+2|V_{33}|=e(G)\equiv1(\text{mod }2)$, and we conclude that $|V_{23}|\equiv1(\text{mod }2)$. So, $U\neq\emptyset$.

Now let us count the number $N$ of $4$-cycles lying within the first copy of $Q_{n-1}$ that pass through $v$, have two vertices colored by $1$, one vertex colored by $2$ and one vertex colored by $3$. Since every $z\in V_{23}$ belongs to exactly one such $4$-cycle, we conclude that $N=|V_{23}|\equiv1(\text{mod }2)$. On the other hand, for every vertex $u$  in $U$  and every vertex $w$ in $W$, there is a unique neighbor $z$ in the second neighborhood of  $v$   and this $z$  must be colored $1$, hence $z$ is a member of $V_{23}$.
By counting pairs $(u,w)$ in $(U,W)$ that create $(1,2,3,1)$-colored $4$-cycles passing through $v$, we conclude that there are precisely  $|U||W| = N$  such pairs. However, as $|U|\equiv0(\text{mod }2)$, we have $|U||W|\equiv0(\text{mod }2)$, which contradicts the already established $N\equiv1(\text{mod }2)$.
\end{proof}

\noindent\textbf{Remark.} Recall that $N(v)$ (resp. $N[v]$) denotes the open (resp. closed) neighborhood of a vertex $v$ in $G$. Similarly, $N_2(v)$ is the second neighborhood of $v$. Let us call
a bipartite graph $G$ \textit{nice} if for a vertex $v$ in $G$ holds: $(i)$ every vertex in $N[v]$ is of odd degree; $(ii)$ for every $u,w\in  N(v)$, $|N(u)\cap N(w)| = 2$; and $(iii)$ for every $z\in N_2(v)$, $|N(z)\cap N(v)| = 2$.
The proof of Theorem~\ref{hypercube} applies verbatim to the following more general result: \textit{If $G$ is a nice graph, then $\chi_o(G\Box K_2) = 4$.}

%--------------------------------------------------------------------------------------------------------------------------------

\section{Upper bounds on the odd chromatic number}

In this final section we give and conjecture several upper bounds for the odd chromatic number. In the first third we confine to planar graphs and revisit a conjecture recently posed in~\cite{PetSkr22}. The second third deals with certain degenericity aspects. Finally, in the last third we discuss bounding the odd chromatic numbers by a linear function of the maximum degree.

\paragraph{Planar graphs.} The following has been proven in~\cite{PetSkr22}.

\begin{theorem}
    \label{t.9colors}
 If $G$ is a connected planar graph, then $\chi_o(G)\leq 9$.
   \end{theorem}
\noindent

Also there was posed the following conjecture.

\begin{conjecture}
    \label{conj:1}
For every planar graph $G$ it holds that $\chi_o(G)\leq 5$.
\end{conjecture}

We prove here this conjecture for outerplanar graphs. Note that the reduction of 2-vertices in the proof bellow
is taken from~\cite{PetSkr22}.

\begin{proposition}
\label{p.outerplanar}
For every outerplanar graph $G$ it holds that $\chi_o(G)\leq 5$. \end{proposition}

\begin{proof}
 Suppose $G$ is a minimal counter-example. By Proposition~\ref{trees}, we may assume that $G$ has a cycle, and so it is of order at least $3$. Let $v\in V(G)$ be a vertex of minimum degree. If $d(v)=1$, then take an odd $5$-coloring of $G-v$ by minimality. This coloring extends to an odd $5$-coloring of $G$ since at most two colors are forbidden at $v$: the color of its only neighbor $z$ and possibly another color that is unique in regard to having an odd number of occurrences in $N_{G-v}(z)$. So we may assume that $v$ is of degree $2$, and say $N_G(v)=\{x,y\}$. Construct $G'$ by removing $v$ from $G$ plus connecting $x$ and $y$ if they are not already adjacent. By minimality, $G'$ admits an odd $5$-coloring $c$.  Say $c(x)=1$ and $c(y)=2$, and let color(s) $1'$ and $2'$, respectively, have odd number of occurrences in $N_{G'}(x)$ and $N_{G'}(y)$. If there are more possibilities for $1'$, then choose $1'\neq2$; do similarly for $2'$ in regard to $1$. Extend the coloring to $G$ by using for $v$ a color different from $1,2,1',2'$. The properness of the coloring is clearly preserved. As for the oddness concerning neighborhoods, $v$ is fine because $1\neq2$. If $1'\neq2$ then $x$ is fine since $1'$ remains to be odd in $N_G(x)$. Contrarily, if $1'=2$ then $c(v)$ is odd in $N_G(x)$. Similarly, the vertex $y$ is also fine. The obtained contradiction proves our point.
\end{proof}

As first support to Conjecture~~\ref{conj:1}, it was shown Theorem~\ref{t.9colors}, whose supplied proof used the discharging technique without evoking the Four Color Theorem~\cite{AppHak77,RobSanSeyTho96}. On the other hand, by making use of the Four Color Theorem, the following has been recently shown  (see Theorem 4 in~\cite{AasAkb20}).

\begin{theorem} [Aashtab et al., 2020]
    \label{odd_forest}
Let $G$ be a connected planar graph of even order. Then $V(G)$ partitions into at most $4$ sets such that each partite set induces an odd forest.
\end{theorem}

From Theorem~\ref{odd_forest} and Observation~\ref{o.2parts} it immediately follows that
\begin{corollary}
\label{c.even_order}
If $G$ is a connected planar graph of even order, then $\chi_o(G)\leq8$.
\end{corollary}
\begin{proof}
Every odd forest $F$ is isolate-free and has $\chi_o(F)=\chi(F)=2$. By Observation~\ref{o.2parts}, $\chi_o(G)\leq 2+2+2+2=8$.
\end{proof}

In regard to planar graphs of odd order, we have the following.

\begin{corollary}
\label{c.odd_order}
If $G$ is a connected planar graph of odd order that has a vertex  of degree 2 or any odd degree, then $\chi_o(G)\leq 8$.
\end{corollary}
\begin{proof} Let $v$ be a vertex with degree in $\{1, 2, 3, 5, 7,\ldots\}$.
If $d(v)=2$, then we deal as it is already explained in details in the proof of  Proposition~\ref{p.outerplanar}: delete
$v$ and add an edge between its neighbors if they are not already adjacent; take an odd 8-coloring $c$ of $G'$ by Corollary~\ref{c.even_order}, and extend $c$ to $G$ as we have at least four available colors for $v$.

And, if $v$ is a vertex of odd degree, then attach to it a new leaf $w$. The obtained graph $G'$ is connected and of even order, hence we
can apply Corollary~\ref{c.even_order} and use an odd coloring $c$ of $G'$. Now observe that $c$ is also an odd coloring
of $G$ as therein $v$ is of odd degree and hence the oddness condition at $v$ is satisfied a priori.
\end{proof}

Notice that if one is able to somehow reduce vertices of degree 4, then it would yield the conclusion that every planar graph $G$ satisfies
$\chi_o(G)\leq 8$.  For the time being, we can only reduce them by introducing a new color, say 9, at one 4-vertex. This gives an alternative proof of Theorem~\ref{t.9colors}. The succinct proof below is only several lines long, but on other hand it relies on the Four Color Theorem.

\bigskip
\noindent
{\bf Proof of Theorem~\ref{t.9colors}}
If the graph $G$ is of even order or it has an odd vertex or a vertex of degree 2, then we are done by Corollaries~\ref{c.even_order} and \ref{c.odd_order}. Assume all degrees in $G$ are even  $\ge 4$. As $G$ is planar,  it must have a vertex $v$ of degree 4.
Attach a leaf $w$ to $v$ to get $G^*$ which is planar, connected and of even order.  Hence $\chi_o(G^*) \le 8$, and we may consider an odd $8$-coloring $c$ of $G^*$. Say $c(w) = 1$,  $c(v)= 2$ and let $2^*$ be a color with an odd number of occurrences in $N_{G^*}(v)$. If $1\ne  2^*$, then simply remove $w$ and conclude that $\chi_o(G) \le 8$.  So we may assume $1=2^*$ is the only `odd
color` in' the neighborhood of $v$.
Choose $u\in N_{G}(v)$ and recolor it by setting $c(u) = 9$. Remove $w$. The resulting coloring $c$ of $G$ is clearly a proper coloring. All neighbors
of $u$ (including $v$) have the color $9$ as an odd color in their respective neighborhoods. As for the vertex $u$, it has an `odd color' (induced from $G^*$) in its neighborhood because no neighbor of it was deleted or changed color.\qed
%The proof is "short" (but it is based on the Four Color Theorem).

%Now, we give an alternative proof of Theorem 1.2 from   \cite{PetSkr22}. The proof is "short" (but it is based on the Four Color %Theorem).

%\begin{theorem}
%    \label{t.9colors}
%If $G$ is a connected planar graph, then $\chi_o(G)\leq 9$.
%\end{theorem}
%\noindent
%{\bf Sketch of the proof.}  The proof is by induction on $n$, the order of $G$. By above corollaries, we may assume that  $n$ is  odd and all vertices are of even degree $\ge 4$.  As $G$ is planar, it follows that we have a vertex $v$ of degree 4.
%We use the same argument as the one explained in details in  \cite[see Claim 1 of the proof of Theorem 1.2]{{PetSkr22}}. And, it goes by contracting en edge incident with $v$, say $e=vu$. As we odd-color $G/e$ with colors $1,\ldots,8$, we put a new color, say 9, to $v$. Note that this reduction assures that some color appear uniquly in the neigborhod of $v$ (it is the color of $u$).
%\qed
%----------------------------------------------------------------------------------------------------------------

\paragraph{Degenericity.} Let us consider briefly certain aspects of bounding the odd chromatic number of a graph in regard to its degenericity. Start by noting that the graph $G$ obtained by subdividing once every edge of $K_n$, is a $2$-degenerate bipartite graph having $\chi_o(G) \ge n$.
Hence a natural question is when does a $k$-degenerate graph have its odd chromatic number bounded by a linear
function of $k$, for example,  $\chi_o(G) \le 2k +1$.

\smallskip

Recall the classical notion of a $k$-tree: it is a graph formed by starting with a copy of $K_{k+1}$, and  then repeatedly adding vertices in such a way that each added vertex $v$ has exactly $k$ neighbors $U$ so that, together, the $k + 1$ vertices formed by $v$ and $U$ form a clique. (Thus $1$-trees are the same as trees.) Motivated by this constructive definition of $k$-trees, we introduce a similar notion. Namely, we say that a connected graph $G$ is a {\it half $k$-tree} if $G$ is established in the following way:
\begin{itemize}
\item The initial stage of $G$ is any connected graph $H$ on $p$ vertices, where $2 \le p \le k+1$ and such that $\chi(H)\ge \lfloor p/2\rfloor +1$.
(In particular, $H$ may contain a clique of order at least $\lfloor p/2\rfloor +1$. )
\item  At each intermediate stage, add a new vertex $v$, adjacent to at most $k$ vertices such that the subgraph of the current $G$ induced by $N(v)$
has  chromatic number at least $\lfloor d(v)/2\rfloor +1$, i.e. $\chi(G[N(v)])\ge \lfloor d(v)/2\rfloor +1$. (In particular, $N(v)$ may contain a clique of
cardinality at least $\lfloor d(v)/2\rfloor +1$.)
\end{itemize}
Observe that $k$-trees form a subclass of the newly defined class of half $k$-trees. The reason for calling it `half $k$-tree'  is an obvious adoption inspired by $k$-trees and the fact the for a half $k$-tree,  $N(v)$ induces a subgraph of chromatic number at least  $\frac{d(v)}{2} + 1$ (here is where the adjective `half' comes from whereas in a $k$-tree this neighborhood gives a clique).

\begin{proposition}
If $G$ is a half $k$-tree, then $\chi_o(G) \le 2k +1$.
\end{proposition}
\begin{proof}
The proof goes by induction on the order $n$ of $G$.  For the initial stage $H$, as $V(H) \le k +1$, we have
$\chi_o(H) \le k +1< 2k +1$. Assume induction holds for $n-1$ and let $G$ be a connected half $k$-tree graph
on $n$ vertices. Let $v$ be the last vertex in the process of creating $G$.
Then $d(v) = t \le k$ and $\chi(G[N(v)]) \ge \lfloor d(v)/2\rfloor +1$.
Consider $H = G - v$. By definition $H$ is a half $k$-tree; in particular $H$ is connected (by definition, how these graphs are constructed) and by induction $\chi_o(H) \le 2k +1$.

Now consider the neighbors of $v$, say $u_1,\ldots,u_t$, where clearly $1\le t \le k$ and distinguish between the following
two cases.

\smallskip

\noindent
{\bf Case 1:} {\em $t$ is odd.}
Then at least one of the colors on the vertices in $N(v)$ appears an odd number of times. Since there are $t$ neighbors
of $v$, all the neighbors forbid at most $t$ colors for $v$. Since each of the neighbors has at least one `odd color', such odd colors (one from each neighbor of $v$) may forbid (if unique) another $t$ colors for $v$. Hence at
most $2t \le 2k$ colors are forbidden, and thus we have an available color
for $v$. As $t$ is odd, the coloring is indeed odd,
proving that $\chi_o(G) \le 2k +1$ holds in this case.

\smallskip

\noindent
{\bf Case 2:} {\em $t$ is even.}
Set $t = 2q$. By definition, $\chi(G[N(v)])\ge q +1$ and at least $q+1$ vertices of $N(v)$ receives distinct colors in
the coloring of $H$.
The rest $q-1$ vertices in $N(v)$ give at least two vertices whose colors appear exactly once in $N(v)$.
So $N(v)$ already has an odd coloring.
Now as $v$ has $t \le k$ neighbors and each such neighbor and an odd color related to this neighbor in
$H$ may forbid at most $2$ colors, we conclude that at most $2t \le 2k$ colors are forbidden for $v$. Hence we have a free color for $v$ that extends the odd coloring of $H$ to an odd coloring of $G$ by using at most $2k +1$ colors.
\end{proof}

A graph is {\em maximal $k$-degenerate} if each induced subgraph has a vertex of degree at most $k$ and adding
any new edge to the graph violates this condition. Such a graph can be constructed in the following way \cite{FMS}:
starting from the complete graph $K_{k +1}$, in each subsequent step add a vertex adjacent to precisely $k$ vertices.  The sequence of vertices $v_n,\ldots,v_{k+2}$ that are added sequentially  to $K_{k +1}$ is called the {\em elimination order}.

 \begin{proposition}
    \label{prop odd k}
Let $G$ be a maximal  $k$-degenerate graph where $k$ is odd.  Then $\chi_o(G) \le 2k +1$.
\end{proposition}
\begin{proof}
 The proof goes by induction on $n$. For $n =k +1$, $G = K_{k +1}$ and we have  $\chi_o(G) = \chi(G) =  k+1 \le 2k +1$.

Now, assuming that the claim holds for $n$, we prove it for $n +1$. In order to do so, let $v$ be the last vertex in the elimination order. Then $d(v) = k$. Consider $H = G - v$.  Clearly $H$ is  maximal  $k$-degenerate (this is a hereditary property).
Since $G$ is  $k$-connected,  $H$ is maximal $k$-degenerate  $k$-connected  (hereditary for $n\ge k +1$, see \cite[Proposition 2]{FMS}) and $\chi_o(H) \le  2k +1$ by induction.
Consider $N(v)$ in $G$. Since there are $k$ vertices in $N(v)$ and $k$ is odd, one of the colors must appear an odd number of times. Also $v$ has at most $2k$ forbidden colors imposed by its neighbors and their `odd color' classes. So there is an available color for $v$ which completes an odd coloring of $G$ with at most $2k +1$ colors.
\end{proof}

Proposition~\ref{prop odd k} supplies a linear upper bound  (of $2k +1$) for the odd chromatic number of maximal $k$-degenerate graphs in case  $k$ is odd.
It is of interest to obtain an upper bound depending only on $k$, in particular a linear upper bound, for the odd chromatic number of maximal $k$-degenerate graphs in case $k$ is even.

%----------------------------------------------------------------------------------------------------------------
%\newpage

\paragraph{Maximum degree.} We start this final portion by showing that, excepting $C_5$, the odd list chromatic number of every other connected graph is at most twice its maximum degree.

\begin{theorem}
\label{p.2D-choosability}
Let $G$ be a connected graph of maximum degree $\Delta\geq1$. Every vertex $v\in V(G)$ is assigned with a color list $L(v)$ of size $2\Delta$. Then $G$ admits an odd coloring $\varphi$ such that $\varphi(v)\in L(v)$ for each $v$, unless all lists are the same and $G=C_5$.
\end{theorem}
\begin{proof}
First we deal with the case when $\Delta\leq2$. If $\Delta=1$ then $G=K_2$ and the assertion follows immediately as the number of vertices equals the length of each list. Assume $\Delta=2$. Thus $G$ is either a path or a cycle and every vertex is assigned with a $4$-list of colors. In case $G$ is a path, a straightforward induction shows that actually $3$-lists as enough (see also the Remark after Proposition~\ref{trees}). Indeed, simply remove a leaf $u$, take an odd coloring of $G-u$ that is compliant with the $3$-lists assignment, delete from $L(u)$ the colors used for its neighbor and second-neighbor along $G$, and then color $u$ by the remaining color in that list.

So let $G$ be a cycle. If all lists are the same, then \eqref{cycles} applies. Consider the situation when not all lists are the same. Then there are adjacent vertices $u,v$ such that $L(u)\neq L(v)$. Pick a color from $L(u)\backslash L(v)$, use it on $u$ and remove this color from all remaining lists. As shown in the previous paragraph, the path $G-\{u,v\}$ admits an odd coloring that is compliant with the resulting lists assignment (as every list is of length at least $3$). Remove from $L(v)$ the colors used in $N(v)$ or $N_2(v)$. At least one color remains on the list since the color of $u$ is not in $L(v)$, and we may use it for $v$.

\smallskip

Now we consider the case of $\Delta\geq3$.
Choose a vertex $v$ of odd degree, if possible, and otherwise of arbitrary degree $d$. Consider a breadth-first search tree rooted at $v$, and
assign indices in decreasing order  $n,n-1,\ldots,1$ to the vertices as we reach them. Thus we obtain an ordering
$v_1,v_2,\ldots, v_n$ of $V(G)$ such that:  $v_n=v$, every $v_i$ with $i<n$ has a higher-indexed neighbor, and
$N(v_n)=\{v_{n-1},\ldots,v_{n-d}\}$, where $d$ is the degree of $v$.

We color the vertices in the order $v_1,v_2,\ldots, v_n$ as follows. First we pick a color $c\in L(v_{n-1})$ and remove it from all lists. Denoting $L'(v_i)=L(v_i)\backslash\{c\}$, we have that $|L'(v_i)|\geq2\Delta-1$. For a pair of adjacent vertices $v_i$ and $v_j$ with $j<i$, we say that $v_i$ is the \textit{terminal neighbor} of $v_j$ if $v_i$ is the highest-indexed neighbor of $v_j$. For $1\leq i\leq n-1$ and vertices $v_1,v_2,\ldots,v_{i-1}$ already colored, remove from $L'(v_i)$ the colors used on the lower-indexed neighbors of $v_i$; additionally, if for any such $v_j$ the vertex $v_i$ happens to be its terminal neighbor and a unique color, say $c_j$, occurs an odd number of times in $N_{G[v_1,v_2,\ldots,v_{i-1}]}(v_j)$ then remove the color $c_j$ from $L'(v_i)$ as well. Noting that at least one color remains in $L'(v_i)$ (as we have removed at most $2(\Delta-1)$ colors), we have an available color that we use for $v_i$. We are left with assigning a color to $v_n$. Similarly to above, for any neighbor $v_j$ of $v_n$ such that a unique color $c_j$ occurs an odd number of times in $N_{G[v_1,v_2,\ldots,v_{n-1}]}(v_j)$, remove $c_j$ from $L'(v_n)$. Also, remove from $L'(v_n)$ the colors of $v_{n-d},v_{n-d+1},\ldots,v_{n-2}$. We have thus removed at most $2\Delta-1=\Delta+(\Delta-1)$ colors from $L'(v_n)$. There are two possibilities.

\smallskip

\noindent \textbf{Case 1:} \textit{A color occurs an odd number of times in $N(v_n)$.} If $c\in L(v_n)$, then color $v_n$ with $c$ and we are done. Contrarily, as $c\notin L(v_n)$, at least one color has remained in $L'(v_n)$ and we use such an available color for $v_n$. Recolor $v_{n-1}$ with $c$, and we are done.

\smallskip

\noindent \textbf{Case 2:} \textit{No color occurs an odd number of times in $N(v_n)$.} Then $d$ and $\Delta$ are even, and at most $d/2$ colors are used for $v_{n-d},v_{n-d+1},\ldots,v_{n-1}$. Therefore, as we are assuming $\Delta\geq3$, we actually have $\Delta\geq4$. Recolor $v_{n-1}$ with $c$. Note that at most $\Delta+d/2$ colors have been removed from $L'(v_n)$. Thus, as $(3/2)\Delta\leq 2\Delta-2$ holds whenever $\Delta\geq4$, a color has remained in $L'(v_n)$. We use such a color for $v_n$ and we are done.
\end{proof}

An immediate consequence of Theorem~\ref{p.2D-choosability} is the following upper bound on the odd chromatic number.

\begin{corollary} For every connected graph $G\ne C_5$, of maximum degree $\Delta$ it holds
                           $$\chi_o(G) \le  2\Delta\,.$$
\end{corollary}

  By \eqref{cycles}, the established bound of $2\Delta$ colors is achieved for every cycle ($\neq C_5$) of length not divisible by $3$. We are not aware of any such examples when $\Delta\geq3$. Our next result sheds light on this issue in the case $\Delta=3$.
%It turns out that, excluding regular graphs, $2\Delta -1$ colors always suffice. In fact, odd coloring from $(2\Delta-1)$-lists is always possible. This is the content of our next result and the proof resembles the one of Proposition~\ref{p.2D-choosability}
%
%
%
%\begin{proposition}
%    \label{p.2D-1-choosability}
%Let $G$ be a non-regular connected graph of maximum degree $\Delta$. Every vertex $v\in V(G)$ is assigned with a color list $L(v)$ of size $2\Delta-1$. Then $G$ admits an odd coloring $\varphi$ such that $\varphi(v)\in L(v)$ for each $v$. \end{proposition}
%
%
%
%\begin{corollary}
%    \label{c.2D-1}
%Let $G$ be a non-regular connected graph of maximum degree $\Delta$. Then $\chi_o(G)\leq 2\Delta-1$.
%\end{corollary}

%Note that this bound (of $2\Delta-1$ colors) is achieved for every path of length $\geq2$. We are not aware of any such examples when $\Delta\geq3$.

\begin{proposition}
    \label{subcubic}
If $G$ is a connected graph of maximum degree $\Delta=3$, then $\chi_o(G)\leq 4$.
\end{proposition}

\begin{proof}
Consider a minimum counter-example $G$. It is not an odd graph, for otherwise $\chi_o(G)=\chi(G)\leq4$. So there is a $2$-vertex in $G$. Moreover, there must exist a $2$-vertex that is adjacent to a $3$-vertex. Say $v$ is the former, $w$ the latter, and let $u$ be the other neighbor of $v$. Delete $v$ and add an edge $uw$ unless  $u,w$ are already adjacent. If the obtained smaller graph $G'$ admits an odd $4$-coloring, then such a coloring extends to $v$ by forbidding the colors used for $u$ and $w$, and a possible third color in regard to odd occurrence in $N_{G'}(u)$ (in case $d_G(u)=2$). Hence $\chi_o(G')\geq5$. Since $G'$ is subcubic, the minimality choice of $G$ implies that $G'$ is of maximum degree $2$ and that $u,w$ were adjacent in $G$. In view of the equalities \eqref{paths} and \eqref{cycles}, we conclude that $G'=C_5$. However, then it is readily checked that $\chi_o(G)=3$, a contradiction.
\end{proof}

\section{Further work}

If $G$ is odd then obviously $\chi_o(G)=\chi(G)$. It might be interesting to look more closely into the class of graphs for which the odd and the ordinary chromatic numbers coincide. Perhaps these graphs exist in abundance, and it is NP-hard to characterize them.

\smallskip

In view of Theorem~\ref{hypercube},
one naturally wonders about odd list colorability aspects of hypercubes.

\begin{question}
Is every hypercube odd $4$-list colorable?
\end{question}

Even more so,

\begin{question}
Does the equality $\chi_o(Q_n)={\rm ch}_o(Q_n)$ hold for every $n$?
\end{question}

Recall that  for any graph $G$,  $\bar{G}$ denotes the complement of $G$, that is, the graph defined on the vertex set of $G$ so that an edge belongs to $\overline{G}$ if and only if it does not belong to $G$. Nordhaus and Gaddum~\cite{NorGad56} studied the chromatic number in a graph  and in its complement  together. They proved sharp lower and upper bounds on the sum and on the product of $\chi(G)$ and $\chi(\overline{G})$  in terms of the order $n$  of $G$. For example, they showed that $\chi(G)+\chi(\overline{G})\leq n+1$. Since then, any bound on the sum and/or the product of an invariant in a graph $G$ and the same invariant in the complement $\overline{G}$ of $G$ is called a Nordhaus-Gaddum type inequality or relation.

\begin{problem}
Let $\mathcal{G}(n)$ denote the class of graphs of order $n$. Given a positive integer $n$, determine (sharp) bounds for $\chi_o(G)+\chi_o(\overline{G})$ and $\chi_o(G)\cdot \chi_o(\overline{G})$, where $G$ ranges over the class $\mathcal{G}(n)$, and characterize the extremal graphs.
\end{problem}

Let us share some of our initial thoughts on this matter. By Nordhaus-Gaddum Theorem,  we know 
$\chi(G)\cdot \chi(\overline{G}) \ge n$, and since $\chi_o(G)  \ge \chi(G)$  it follows  that  
$\chi_o(G)\cdot  \chi_o(\overline{G}) \ge n $. This bound is attained by taking $G =  K_n$. 

%tThus, the product $\chi_o(G)\cdot \chi_o(\overline{G})$ attains minimum $n$, as it is a case with the usual chromatic number.
%First note for an odd-coloring $c$ of $G$ and an odd-coloring $\overline{c}$ of  $\overline{G}$, 
%we obtain that $(c,\overline{c})$ is an odd-coloring of $K_n$. Hence 
%$\chi_o(G)\cdot \chi_o(\overline{G})\ge \chi_o(K_n)=\chi_o(K_n) \cdot \chi_o(\overline{K_n})=n$. 

By~\eqref{b.t}, we have $\chi_o(G)\leq\chi(G)+\gamma_{t}(G)$. Consequently, $\chi_o(G)+\chi_o(\overline{G})\leq (\chi(G)+\chi(\overline{G}))+(\gamma_{t}(G)+\gamma_{t}(\overline{G}))$. As already mentioned, the first summand is bounded from above by $n+1$. In regard to the second summand, we make use of Proposition~$1.5$ in~\cite{HenYeo07}, which reads: If $G$ is a graph of order $n$ and minimum degree $\delta\geq1$ then $\gamma_{t}(G)\leq \frac{n(1+\ln\delta)}{\delta}$. Consequently, for $\min\{\delta(G),\delta(\overline{G})\}\geq10$ we have $\chi_o(G)+\chi_o(\overline{G})\leq\frac{5}{3}n$, and already for $\min\{\delta(G),\delta(\overline{G})\}\geq 6$ we get $\chi_o(G)+\chi_o(\overline{G})\leq 1.94n$. Furthermore, if $\min\{\delta(G),\delta(\overline{G})\}\to\infty$ with $n$, then we are bounded by $n+o(n)$. So for almost all graphs $G$ of order $n$, the sum $\chi_o(G)+\chi_o(\overline{G})$ is at most $n+o(n)$.  We are rather optimistic and believe that the following may hold.

\begin{conjecture}
For every graph $G\neq C_5$ of order $n$ it holds that $\chi_o(G)+\chi_o(\overline{G})\leq n+3$.
\end{conjecture}

If true, the conjectured bound is best possible for every $n\geq6$. Indeed, the graph $G=K_{n-5}\vee C_5$ satisfies that $\chi_o(G)=n-2$ and $\chi_o(\overline{G})=5$. (The latter equality follows from the fact that $G$ can be obtained from $K_n$ by removing the edges of a $5$-cycle.)

\bigskip

In view of Proposition~\ref{subcubic}, we are tempted to end this paper with the following.

\begin{conjecture}
    \label{conj.D+1}
If $G$ is a connected graph of maximum degree $\Delta\geq3$, then $\chi_o(G)\leq \Delta+1$.
\end{conjecture}

\bigskip
\noindent
{\bf Acknowledgements.}
This work is partially supported by ARRS Program P1-0383 and ARRS Project J1-3002.

\bibliographystyle{plain}

\end{document}